\documentclass[a4paper]{article}

\usepackage{biblatex}
\usepackage{subcaption}
\usepackage{url}
\usepackage{listings}
\usepackage{amsthm,amsmath, amssymb}
\usepackage{graphicx}
\usepackage{xcolor}
\usepackage{hyperref}
\usepackage[capitalise,nameinlink]{cleveref}
\usepackage{todonotes}
\usepackage{algorithm}
\usepackage{algpseudocode}
\usepackage{enumerate}

\hypersetup{
	colorlinks=true,
    linkcolor={red!50!black},
    citecolor={blue!50!black},
    urlcolor={blue!80!black},
	bookmarksopen=true,
	bookmarksnumbered,
	bookmarksopenlevel=2,
	bookmarksdepth=3
}

\newtheorem{theorem}{Theorem}

\newtheorem{lemma}[theorem]{Lemma}

\newtheorem{corollary}[theorem]{Corollary}

\newtheorem{observation}[theorem]{Observation}

\newcommand{\mex}{\mathrm{mex}}
\newcommand{\sg}{\mathrm{SG}}

\title{Some results on LCTR, an impartial game on partitions\thanks{This work is supported in part by the Slovenian Research Agency (research program P1-0383, research projects N1-0160, N1-0209 and J3-3003, and bilateral project BI-US/22-24-164).}}

\author{Eric Gottlieb \and Jelena Ilić \and Matjaž Krnc}

\begin{document}
\maketitle
\begin{abstract}

We apply the Sprague-Grundy Theorem to LCTR, a new impartial game on partitions in which players take turns removing either the Left Column or the Top Row of the corresponding Young diagram. We establish that the Sprague-Grundy value of any partition is at most $2$, and 
determine Sprague-Grundy values for several infinite families of partitions.
Finally, we devise a dynamic programming approach which, for a given partition $\lambda$ of $n$, determines the corresponding Sprague-Grundy value in $O(n)$ time. 
\end{abstract}

\section{Introduction}

Integer partitions are widely-studied objects of interest to researchers in combinatorics, number theory, representation theory, and physics. Combinatorial Game Theory is a growing discipline concerned with the study of two-player games with perfect information and no randomization (see \cite{
andrews1998theory,
berlekamp2018winning1,
berlekamp2018winning2,
berlekamp2018winning3,
berlekamp2017winning4,
siegel2013combinatorial}). 
The aim of this discipline is to develop, study, and apply mathematical tools for analyzing such games with the ultimate goal of determining which player can force a win (see \cite{Gru39,Spr35,Spr37}). 

In this paper, we describe and analyze LCTR, an impartial game that is played on (integer) partitions. In \cref{prelim}, we define some terms, notation, and theorems concerning combinatorial game theory and partitions. \cref{sec:LCTR} contains the main results of our work. Here, we will describe the game LCTR, determine Sprague-Grundy values for certain infinite families of partitions, and prove some complexity results. In \cref{conc} we offer some directions for future study. 

\section{Preliminaries}\label{prelim}

In this section we provide background information on combinatorial game theory and partitions. 

\subsection{Combinatorial game theory}

Combinatorial games are a form of strategic interaction between two players. They are \textit{deterministic}, meaning the outcome depends only on the moves chosen by the players and not on randomizing elements like dice, deals from a card deck, flips of a coin, and the like. They assume \textit{perfect information}, which means that both players know the position of the game at all times and what moves are available to the players from any given position. 

Starting from some initial position, players take turns choosing a move from a set of allowable moves until the game reaches a \textit{terminal position}, i.e., a position from which no legal move remains. Under \textit{normal play}, the last player to make a legal move wins. Under \textit{mis\'ere play}, the last player to make a legal move loses. In this paper, we focus exclusively on LCTR under normal play. 

A combinatorial game is called \textit{impartial} if the set of allowable moves depends only on board position and not on the player making the move; otherwise it is called \textit{partisan}. For purposes of this discussion, we assume that games guaranteed to conclude after a finite number of moves from any position. LCTR is impartial and terminates after finitely many moves. 

Let $\mathfrak P(A)$ denote the power set of the set $A$. Formally, an impartial combinatorial game is a pair $G = (X, F)$. Here $X$ is a set whose elements are referred to as \textit{positions} and $F: X \rightarrow \mathfrak P(X)$ is a function that describes the legal moves of $G$. That is, $y$ is a legal move from $x$ if and only if $y \in F(x)$. Terminal positions are those $x$ for which $F(x) = \emptyset$. If $y \in F(x)$ then we say that $y$ is a \textit{follower} of $x$. 

Impartial games under normal play are subject to study via Sprague-Grundy theory, which associates a nonnegative integer to each position of the game. This is done in such a way that the previous player has a winning strategy precisely when the associated number is zero. Thus, the Sprague-Grundy theorem refines the notion of winning and losing positions. It is a powerful tool that is not available when studying partisan games or games under mis\'ere play (see~\cite{siegel2013combinatorial}).

We say that a player has a \textit{winning strategy} if they can win no matter what moves their opponent makes. An \textit{$\mathcal{N}$-position} is a position in which the next player to move has a winning strategy. A \textit{$\mathcal{P}$-position} is a position in which the previous player has a winning strategy. Every position in any game is in exactly one of $\mathcal{N}$ and $\mathcal{P}$. We say that $\mathcal P$-positions are \textit{losing} and $\mathcal N$-positions are \textit{winning}. 

\begin{theorem}\label{thm0}
  The $\mathcal{N}$- and $\mathcal{P}$-positions of an impartial combinatorial game under normal play are characterized by the following properties.
  \begin{itemize}
    \item A position is in $\mathcal{N}$ if and only if there is at least one move to a $\mathcal{P}$-position.
    \item A position is in $\mathcal{P}$ if and only if there are no moves to a $\mathcal P$-position.
  \end{itemize}
\end{theorem}
The second condition implies that every terminal position of a  impartial normal-play game is a $\mathcal P$-position. It also allows us to compute $\mathcal P$-positions recursively. For impartial games under mis\`ere play, terminal positions are defined to be $\mathcal{N}$-positions and non-terminal positions are determined as above. More on this theorem can be found in ~\cite{splet6}. 

The Sprague-Grundy theorem provides a function, called the \textit{Sprague-Grundy function}, from the positions of a game into the set of positive integers. This function has the property that it is zero on a position if and only if that position belongs to $\mathcal P$. 

Let $\mathfrak F(A)$ denote finite subsets of a set $A$. The \textit{minimum excluded value function} $\mex:\mathfrak F(\mathbb N) \rightarrow \mathbb N$ plays an important role in the definition of the Sprague-Grundy function. It is defined by \[ \mex(B) =  \min (\mathbb N \setminus B). \]
For brevity and clarity we omit set brackets, e.g. $\mex(\{0, 1, 3\}) = \mex(0, 1, 3) = 2$.

The Sprague-Grundy function $SG:X \rightarrow \mathbb N$ of $G$ is defined recursively for $x \in X$ by
  \begin{align}
    \sg(x) = \mex(F(x)).\label{eq:SG}
  \end{align}
We refer to $\sg(x)$ as the \textit{Sprague-Grundy value} of $x$. For each terminal position $x$ we have $\sg(x) = 0$ because $F(x)$ is empty. 

The function $F$ induces a directed graph $D_G = (X, E)$, where $(x, y) \in E$ if and only if $y \in F(x)$. For our purposes $|F(x)| < \infty$ for all $x \in X$. Observe that $G$ is guaranteed to terminate in finitely many moves precisely when every path in $D_G$ rooted at any $x \in X$ is of finite length. 

For example, the leftmost image in Figure~\ref{dgraph} shows a directed graph. The vertices and edges correspond to  positions and allowable moves, respectively. The center image shows the Sprague-Grundy values associated to each position. The third image shows the classification of positions as $\mathcal N$ or $\mathcal P$.  

  \begin{figure}
    \centering
    \captionsetup{justification=centering}
    \includegraphics[width=0.7\textwidth]{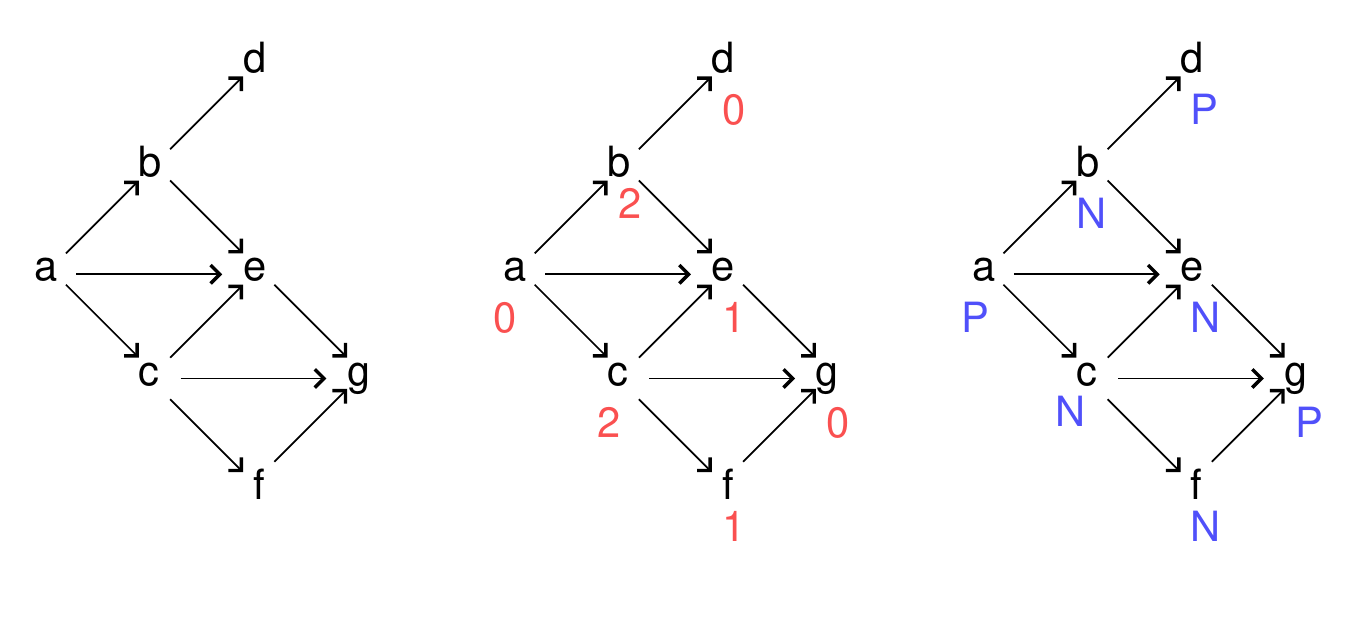}
    \caption{Application of Sprague-Grundy function} \label{dgraph}
  \end{figure}

The Sprague-Grundy values are determined by repeatedly seeking $x \in X$ so that each $y \in F(x)$ is labeled. Thus, in the first iteration, we label the terminal positions $d$ and $g$ with zero. Then, all of the followers of the positions $e$ and $f$ are labeled. The set of Sprague-Grundy values of their followers is $\{0\}$ in both cases, and $\mex(0) = 1$, so they both receive the label 1. At this point, the followers of the positions $b$ and $c$ are all labeled. In both cases, the set of Sprague-Grundy values of their followers is $\{0, 1\}$, so they both receive the label 2 since $\mex(0, 1) = 2$. Finally, the root position $a$ is labeled with 0 since the set of labels of its followers is $\{1, 2\}$ and $\mex(1,2) = 0$.

The $\mathcal{N}$ and $\mathcal{P}$ labels are computed similarly. We repeatedly seek $x \in X$ so that each $y \in F(x)$ is labeled. Thus, the terminal positions $d$ and $g$ are first labeled with $\mathcal{P}$. Now all of the followers of the positions $e$ and $f$ in the center image are labeled. The set of labels of their followers is $\{\mathcal P\}$ in both cases, so these positions receive label $\mathcal N$. Now all followers of the positions $b$ and $c$ have labels; in both cases, the set of labels of their followers is $\{\mathcal N, \mathcal P\}$, so they are labeled $\mathcal N$. The only position that is left to label is the root position $a$, and since the set of labels of its followers is $\{\mathcal{N}\}$, it is labeled with $\mathcal{P}$. Observe that the positions labeled with 0 are precisely those labeled with $\mathcal P$. 

\subsection{Partitions and Young diagrams}

For notation and fundamental properties of partitions we follow the textbook by Andrews \cite{andrews1998theory}.
Given a natural number $n$, a \textit{partition} $\lambda$ of $n$ is just a way of writing $n$ as a sum of positive integers. We refer to $n$ as the \textit{size} of $\lambda$ and write $\lambda \vdash n$. The summands of a partition are called \textit{parts}, and their order does not matter. By convention they are written in descending order. We represent partitions with ordered tuples, using exponentiation to indicate repeated parts. For example, the partitions of 5 are 
    \[ 5 = (5), \quad 4+1 = (4, 1), \quad 3+2 = (3, 2), \quad 3+1+1 = \left( 3, 1^2 \right),\]
    \[2+2+1 = \left( 2^2, 1 \right), \quad 2+1+1+1 = \left( 2, 1^3 \right), \quad  1+1+1+1+1 = \left( 1^5 \right). \]
There is only one partition of zero, namely $()$, i.e. the empty partition. 

Partitions can be visualized by using Young diagrams. The Young diagram of a partition $\lambda_1 + \lambda_2 + \cdots + \lambda_k$ is a left-justified array with $\lambda_i$ squares in the $i$-th row, $i = 1, \dots, k$. For example, the Young diagram of $\left( 5, 3^2, 2, 1^2 \right)$ appears in Figure~\ref{fig:fer}.
\begin{figure}
\centering
\includegraphics[width=0.2\textwidth]{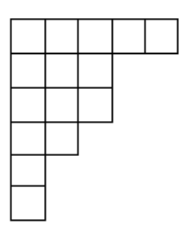}
\caption{The Young diagram of $\left( 5, 3^2, 2, 1^2 \right)$}
\label{fig:fer}
\end{figure}
We refer to a partition and its Young diagram interchangeably. 

The \textit{conjugate} $\lambda'$ of a partition $\lambda = (\lambda_1, \ldots, \lambda_k)$ is obtained from $\lambda$ by exchanging the rows and columns of $\lambda$. More formally, $\lambda' = (\lambda_1', \dots, \lambda_{\lambda_1}')$, where $\lambda_j' = |\{m : \lambda_m \geq j\}|$, i.e., $\lambda_j'$ is the number of parts of $\lambda$ not smaller than $j$. A partition $\lambda$ is \textit{self-conjugate} if $\lambda' = \lambda$. It is well-known that $(\lambda')' = \lambda$.

\section{LCTR \label{sec:LCTR}}

The positions of LCTR are partitions. Given some initial partition $\lambda$, the players take turns by removing the left column or top row of $\lambda$. We use the normal play convention, so the last player with a legal move wins. The only terminal position is the empty partition.

More formally, let $\mathbb P = \{ \lambda \vdash n : n \in \mathbb N\}$. If $\lambda = (\lambda_1, \dots, \lambda_k) \neq ()$, define $T(\lambda) = (\lambda_2, \dots, \lambda_k)$ and $L(\lambda) = (\lambda_1 - 1, \lambda_2 -1, \ldots)$, where nonpositive values are omitted. Then $LCTR = (\mathbb P, F)$, where $F(\lambda) = \{L(\lambda), T(\lambda)\}$ for $\lambda \neq ()$ and $F(()) = \emptyset$. 

An illustration of the computation of the Sprague-Grundy value of the board from \cref{fig:fer} is shown in \cref{fig:computingSG}.
Note that the Sprague-Grundy value of each partition in the figure is at most 2. This is true in general.

\begin{lemma}
 The Sprague-Grundy value in LCTR of any partition does not exceed $2$.
\end{lemma}

\begin{proof}
Let $G = (X, F)$ be any impartial game. For $x \in X$ we have $\sg(x) \leq |F(x)|$ by definition of mex. Since the number of moves from any position in LCTR is at most 2, the result follows. 
\end{proof}

\begin{figure}[h]
\centering
\includegraphics[width=0.6\textwidth]{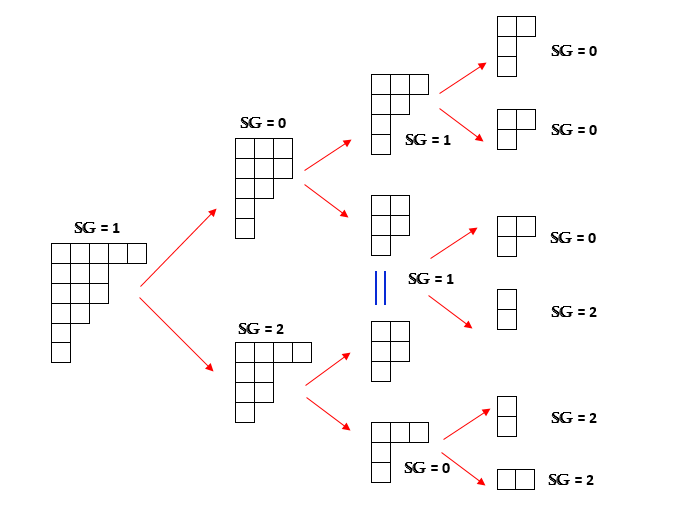}
\caption{Computing the Sprague-Grundy value of $\left( 5, 3^2, 2, 1^2 \right)$. See~\cref{thickGammasect,rectsect} for an explanation of the Sprague-Grundy values of the leaves.
}\label{fig:computingSG}
\end{figure} 

We say that a game on partitions is \textit{conjugate-invariant} if $\sg(\lambda') = \sg(\lambda)$ for every partition $\lambda$. We will use the following lemma to prove that LCTR is conjugate-invariant. 

\begin{lemma} \label{movesconjlem}
Let $\lambda$ be a nonempty partition. Then $L(\lambda) = (T(\lambda'))'$ and $T(\lambda) = (L(\lambda'))'$.
\end{lemma}

\begin{proof}
Let $\lambda = (\lambda_1, \dots, \lambda_k)$. To establish the first equality, note that 
\begin{align*} 
(T(\lambda'))' 
&= (T((\lambda_1, \ldots, \lambda_k)'))' \\
&= (T(\lambda_1', \ldots, \lambda_{\lambda_1}'))' \\
&= (\lambda_2', \ldots, \lambda_{\lambda_1}')' \\
&= (\lambda_1'', \ldots, \lambda_{\lambda_2'}''),
\end{align*}
where $\lambda_j''$ is the number of parts of $(\lambda_2', \ldots, \lambda_{\lambda_1}')$ that are not smaller than $j$. This is equal to $\lambda_j-1$, so 
\begin{align*}
(T(\lambda'))' 
&= (\lambda_1'', \ldots, \lambda_{\lambda_2'}'') \\
&= (\lambda_1-1, \ldots, \lambda_k-1) \\
&= L(\lambda_1, \ldots, \lambda_k).
\end{align*}

To establish the second equality, make the change of variable $\lambda \mapsto \lambda'$ in the first equality; this gives 
    \[ L(\lambda') = (T(\lambda''))' = (T(\lambda))'.\]
Taking the conjugate of both sides gives 
    \[ (L(\lambda'))' = (T(\lambda))'' = T(\lambda),\] 
as desired. 
\end{proof}

\begin{lemma} \label{conjinvlem}
LCTR is conjugate-invariant. 
\end{lemma} 

\begin{proof}
Let $\lambda$ be a partition. We proceed by induction on the size of $\lambda$. The fact that $()$ is self-conjugate establishes the base case. If $\lambda$ has positive size, then by Lemma~\ref{movesconjlem} we have 

\begin{align*}
    \sg(\lambda) &= \mex(\sg(L(\lambda)), \sg(T(\lambda) )) 
    = \mex(\sg((T(\lambda'))'), \sg((L(\lambda'))')) \\
    &= \mex(\sg(T(\lambda')), \sg(L(\lambda'))) 
    = \sg(\lambda'). 
\end{align*}
\end{proof}

\subsection{Sprague-Grundy values for certain families of partitions}

For some special partitions $\lambda$ we can determine which player has a winning strategy, i.e. whether $\lambda$ is in $\mathcal{P}$ or $\mathcal{N}$, and determine their Sprague-Grundy values. For brevity we sometimes omit one set of parentheses in $\sg(\cdot)$, e.g. we write $SG\left( \left( 2, 1^3 \right) \right)$ as $SG\left( 2, 1^3 \right)$. 

\subsubsection{Rectangles} \label{rectsect}

\textit{Rectangles} are partitions of the form $(n^m)$, for positive integers $n$ and $m$. We now determine the Sprague-Grundy value of rectangles of arbitrary size. 

\begin{lemma} \label{rowlem}
Let $n > 0$. Then $(n)$ and $\left( 1^n \right)$ are $\mathcal{N}$-positions. Furthermore, \[ \sg(n) = SG\left( 1^n \right) = \left\{ \begin{array}{rl}
    1 & \mbox{if $n$ is odd}\\
    2 & \mbox{if $n$ is even}
\end{array} \right..\]
\end{lemma}

\begin{proof} 
$(n)$ consists of one row. The next player can win by removing that row, so $(n) \in \mathcal{N}$. 

For $n = 1$, both legal moves result in the empty partition, so $\sg(1) = \mex(0) = 1$. For $n = 2$, we can either remove the whole row which will give Sprague-Grundy value 0, or the leftmost square leaving $(1)$ which has Sprague-Grundy value 1. \\
Hence $\sg(2) = \mex(0, 1) = 2$. 

For $n > 2$, removing the entire row gives the empty partition, while removing the the left column gives $(n-1)$. By induction, we have \begin{align*}
    \sg(n) = \mathrm{mex}(0, \sg(n-1)) &= \left\{ \begin{array}{rl}
        \mathrm{mex}(0, 2) & \mbox{if $n-1$ is even}\\
        \mathrm{mex}(0, 1) & \mbox{if $n-1$ is odd}
    \end{array} \right. \\
    &= \left\{ \begin{array}{rl}
        1 & \mbox{if $n$ is odd}\\
        2 & \mbox{if $n$ is even}
    \end{array} \right.. 
\end{align*} Analogous results follow for $\left( 1^n \right)$ by Lemma~\ref{conjinvlem}. 
\end{proof}

We now introduce notation that allows us to visualize the situation for $(n)$ and other families of partitions. Consider a left-infinite row of squares. By considering any particular square and disregarding the ones to the left, we obtain a partition of the form $(n)$. In that square, we record $\sg(n)$, which is computed by taking the mex of the values to the right and below the square, or zero if no such value is present. Such diagrams record the Sprague-Grundy values of all positions reachable from the given one. See Figure~\ref{rowfig}. 

\begin{figure}[h]
\centering
\includegraphics[width=0.3\textwidth]{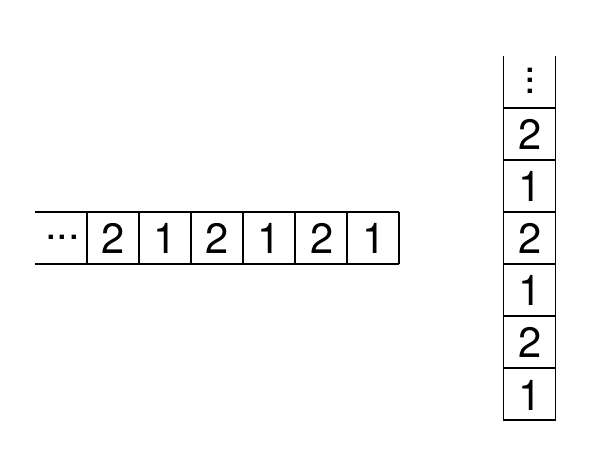}
\caption{Sprague-Grundy values for $(n)$ and $(1^n)$, where $n > 0$} \label{rowfig}
\end{figure} 

We now consider partitions of the form $\left( n^2 \right)$ and $\left( 2^n \right)$. A visual summary for $\left( n^2 \right)$ appears in Figure~\ref{2rowfig}. 

\begin{figure}[h]
\centering
\includegraphics[width=0.3\textwidth]{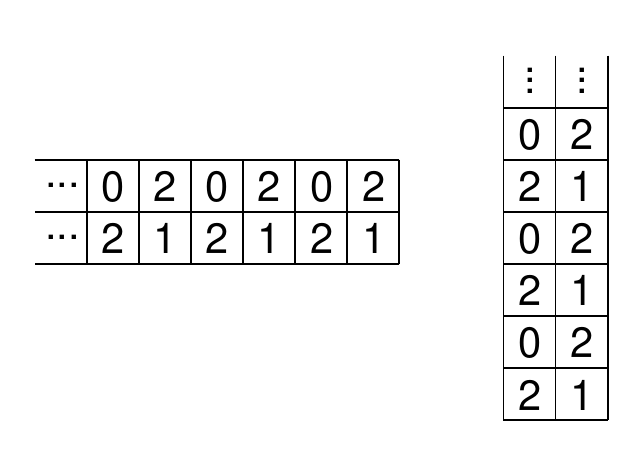}
\caption{Sprague-Grundy values for $\left( n^2 \right)$ and $(2^n)$, where $n > 0$} \label{2rowfig}
\end{figure} 

\begin{lemma} \label{2rowlem}
Let $n$ be a positive integer. Then 
$$SG\left( n^2 \right) = SG\left( 2^n \right) = \left \{ \begin{array}{rl}
    2 & \mbox{if } n \mbox{ is odd} \\
    0 & \mbox{if } n \mbox{ is even}
\end{array} \right..$$
\end{lemma}

\begin{proof}
 We proceed by induction on $n$. By Lemma~\ref{rowlem}, we have $SG\left( 1^2 \right) = 2$. Both moves on $\left( 2^2 \right)$ give a partition with Sprague-Grundy value $2$, so $SG\left( 2^2 \right) = \mex(2) = 0$. These establish the base cases. 
 
 Now suppose that $n > 2$. Then by induction and Lemma~\ref{rowlem}, we have 
\begin{align*}
    SG\left( n^2 \right) = \mex\left( SG\left( ( n-1)^2 \right), \sg(n) \right) 
     = &\left\{ \begin{array}{rl}
     \mex(1, 0) & \mbox{if $n-1$ is even} \\
     \mex(2, 2) & \mbox{if $n-1$ is odd}  
     \end{array} \right. \\
    = &\left\{ \begin{array}{rl} 
    2 & \mbox{if $n$ is odd} \\
    0 & \mbox{if $n$ is even}
    \end{array} \right..
\end{align*}
Lemma~\ref{conjinvlem} now establishes the result for $\left( 2^n \right)$. 
\end{proof}

Next we consider partitions of the form $\left( n^3 \right)$ and $\left( 3^n \right)$. A visual representation of this family of partitions and their Sprague-Grundy values can be found in Figure~\ref{3rowfig}. 

\begin{figure}[h]
\centering
\includegraphics[width=0.3\textwidth]{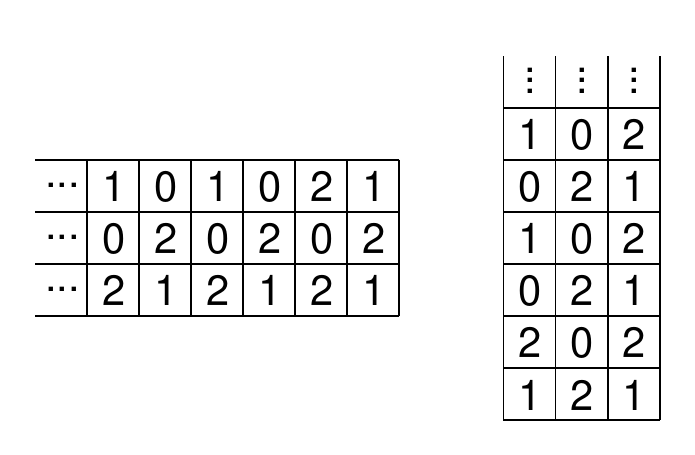}
\caption{Sprague-Grundy values for $\left( n^3 \right)$ and $(3^n)$, where $n > 0$} \label{3rowfig}
\end{figure} 

\begin{lemma} \label{3rowlem}
Let $n \geq 3$. Then 
\[SG\left( n^3 \right) = SG\left( 3^n \right) = \left\{ \begin{array}{rl}
    0 & \mbox{if $n$ is odd} \\
    1 & \mbox{if $n$ is even}
\end{array} \right..\] 
\end{lemma}

\begin{proof}
 By induction on $n$. Lemma~\ref{2rowlem} gives 
    \[SG\left( 3^3 \right) = \mex\left( SG\left( 3^2 \right), SG\left( 2^3 \right) \right) = \mex(2, 2) = 0\]
and \[ SG\left( 4^3 \right) = \mex\left( SG\left( 4^2 \right), SG\left( 3^3 \right) \right) 
    = \mex(0, 0) 
    = 1.\] 
Lemma~\ref{conjinvlem} now gives $SG\left( 3^4 \right) = 1$.

By Lemma~\ref{2rowlem} and induction we have 
\begin{align*}
SG\left( n^3 \right) = \mex\left( SG\left( n^2 \right), SG\left( (n-1)^3 \right) \right) =& \left\{ \begin{array}{rl} 
\mex(2, 1) & \mbox{if $n-1$ is even} \\
\mex(2, 0) & \mbox{if $n-1$ is odd}
\end{array} \right. \\
=& \left\{ \begin{array}{rl}
0 & \mbox{if $n$ is odd} \\
1 & \mbox{if $n$ is even} 
\end{array} \right..
\end{align*}
This establishes the result for $\left( n^3 \right)$; Lemma~\ref{conjinvlem} gives the result for $\left( 3^n \right)$. 
\end{proof}
Next, we consider the case of a rectangle $\left( m^n \right)$, where $m$ and $n$ are not less than 3. This case is illustrated in Figure~\ref{rectfig}, in which a left- and up-infinite square array is shown. For each square, we get a rectangle by ignoring those squares to its left or above. We place the Sprague-Grundy value of this rectangle in the chosen square. 

\begin{figure}[h]
\centering
\includegraphics[width=0.2\textwidth]{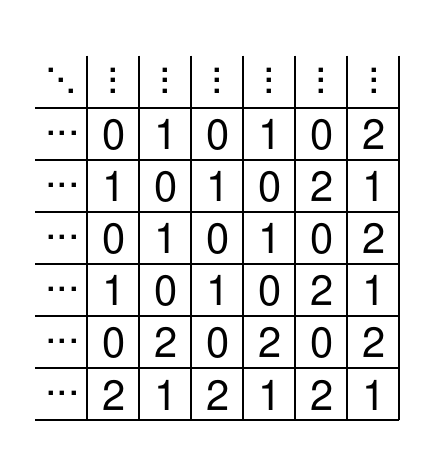}
\caption{Sprague-Grundy values for $\left( n^m \right)$} \label{rectfig}
\end{figure} 

\begin{lemma} \label{rectlem}
Let $m$ and $n$ be integers not less than 3. Then \[ SG\left( n^m \right) = \left\{ \begin{array}{rl}
    0 & \mbox{if $m+n$ is even} \\
    1 & \mbox{if $m+n$ is odd}
\end{array} \right.. \]
\end{lemma}

\begin{proof} Lemma~\ref{3rowlem} establishes the base cases $\left( n^3 \right)$ and $\left( 3^n \right)$. Now suppose that $m$ and $n$ are both greater than three. By induction on $m+n$ we have

\begin{eqnarray*} 
SG\left( n^m \right) &=& \mex\left( SG\left( n^{m-1} \right), SG\left( (n-1)^m \right) \right) \\
&=& \left\{ \begin{array}{rl}
    \mex(1, 1) & \mbox{if $m+n-1$ is odd} \\
    \mex(0, 0) & \mbox{if $m+n-1$ is even}
\end{array} \right.\\
&=& \left\{ \begin{array}{rl}
    0 & \mbox{if $m+n$ is even} \\
    1 & \mbox{if $m+n$ is odd} 
\end{array} \right.,
\end{eqnarray*}
as desired.
\end{proof}

\subsubsection{Quadrated partitions}

A partition is said to be \textit{quadrated} if all of its parts are even and each part occurs an even number of times. An example of a quadrated partition appears in Figure~\ref{quadfig}. Again we place in each square the Sprague-Grundy value of the partition obtained by ignoring squares above or to the left of that square. 

\begin{figure}[h]
\centering
\includegraphics[width=0.35\textwidth]{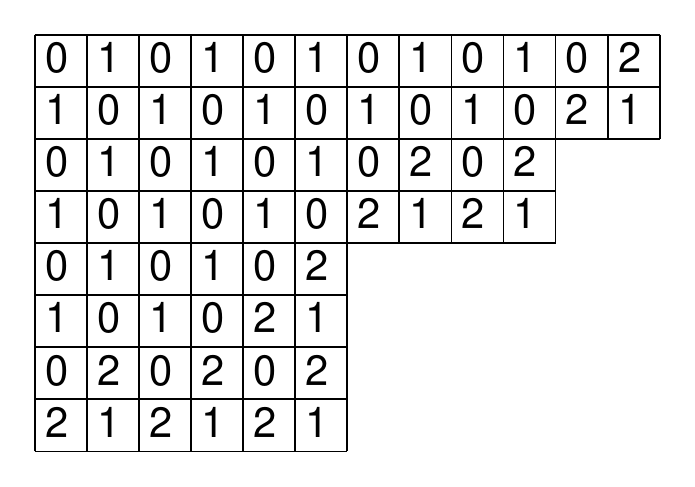}
\caption{Sprague-Grundy values and pendant rectangles for the quadrated partition $\left( 12^2 ,10^2 ,6^4 \right)$} \label{quadfig}
\end{figure} 

\begin{lemma} \label{quadlem}
Let $\lambda$ be a quadrated partition. Then $\lambda$ is a $\mathcal P$-position. Further, if $L(T(\lambda)) = T(L(\lambda))$ is not of the form $(n)$ or $\left( 1^n \right)$, then it is also a $\mathcal P$-position. Hence, $\lambda$ and such $L(T(\lambda))$ have Sprague-Grundy value zero. 
\end{lemma}
\begin{proof}
We proceed by induction on the sum of the number of rows and columns of $\lambda$. The empty partition is a $\mathcal P$-position, so suppose that $\lambda \vdash n > 0$. 

If the first player takes the left column, then the second player can do the same, since each part is even so there must be another column to take. On the other hand, if the first player takes the top row, then the second player can do the same, since each part appears an even number of times so there must be another row. 

Either way, the first player is left with a quadrated partition, which by induction hypothesis is a $\mathcal P$-position. Since the second player has a winning response to every move of the first player, we see that $\lambda$ is a $\mathcal P$-position. 

If $\lambda = \left( (2n_1)^{2m_1}, (2n_2)^{2m_2}, \ldots, (2n_k)^{2m_k} \right)$ then \[ L(T(\lambda)) = \left( (2n_1-1)^{2m_1-1}, (2n_2-1)^{2m_2}, \ldots, (2n_k-1)^{2m_k} \right). \]
If $m_1 = n_2 = 1$ then $\sg(L(T(\lambda)))=0$ by \cref{gammalem}. 
If $m_1 > 1$ or $n_2 > 1$ then whatever move the first player makes, the second player can make the opposite move. Either way, the result is a quadrated partition, which is a $\mathcal P$-position by induction. Thus $L(T(\lambda))$ is a $\mathcal P$-position. 
\end{proof}

We now determine the Sprague-Grundy values partitions reachable from a quadrated partition $\lambda$ in one move, i.e., those of the form $L(\lambda)$ or $T(\lambda)$. 

\begin{lemma}
Let $\lambda  = \left( (2n_1)^{2m_1}, \ldots, (2n_k)^{2m_k} \right)$ be a quadrated partition. If $n_1 = 1$ or if $m_1 = 1$ and $k = 1$, then $\sg(L(\lambda))$ and $\sg(T(\lambda))$ are determined by \cref{2rowlem}. If $n_1 > 1$ and if $m_1 > 1$ or $k > 1$, then $\sg(L(\lambda)) = \sg(T(\lambda)) = 1$. 
\end{lemma}
\begin{proof}
We have 
$
\sg(L(\lambda)) 
= \mex(\sg(L(L(\lambda))), \sg(T(L(\lambda)))).
$
Observe that 
\begin{align*}
    \sg (L(L(\lambda)))&=\sg \left( 2(n_1-1))^{2m_1}, \ldots, (2(n_k-1))^{2m_k} \right)=0\quad\text{and}\\ 
    \sg(T(L(\lambda)))&= \sg \left( (2n_1)^{2(m_1-1)}, \ldots, (2n_k)^{2m_k} \right) =0,
\end{align*}
so $
\sg(L(\lambda)) = \mex \left( 0,0 \right) = 1. $
\end{proof}

\subsubsection{Staircase partitions} 

We define a \textit{staircase} partition to be one of the form $s_n = (n, n-1, \ldots, 2, 1)$. The example is shown in Figure~\ref{stairfig} suggests the following lemma. 

\begin{figure}[h]
\centering
\includegraphics[width=0.2\textwidth]{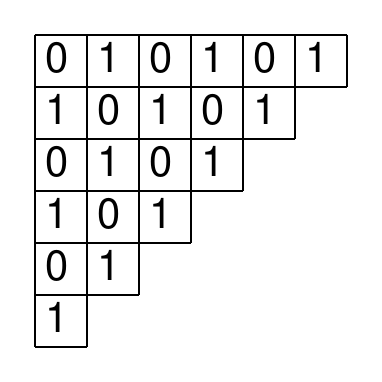}
\caption{Sprague-Grundy values for $\left(6, 5, 4, 3, 2, 1 \right)$} \label{stairfig}
\end{figure}

\begin{lemma}
$\sg(s_n)$ is 1 if $n$ is odd and 0 if $n$ is even. 
\end{lemma}
\begin{proof}
$\sg(s_1) = \sg(1) = 1$. Also, removing the top row or left column of $s_n$ gives $s_{n-1}$, so $\sg(s_n) = \mex(\sg(s_{n-1}))$. By induction, $\sg(s_{n-1})$ will be 0 when $n-1$ is even and 1 when $n-1$ is odd. Thus $\sg(s_n)$ will be 1 when $n$ is odd and 0 when $n$ is even. 
\end{proof}

\subsubsection{Thick $\Gamma$-partitions} \label{thickGammasect}

We define a \textit{thick $\Gamma$-partition} to be one of the form $(n^m, r^s)$, where $n > r$ and $s > 0$. A \textit{$\Gamma$-partition} is a thick $\Gamma$-partition for which $m = r = 1$. 

\begin{lemma} \label{gammalem}
$\Gamma$-partitions are $\mathcal{P}$-positions. Hence, their Sprague-Grundy values are zero. 
\end{lemma}

\begin{proof}
The second player can win by making the opposite move of the first player. Thus $(n, 1^s) \in \mathcal P$ and $\sg(n, 1^s) = 0$ whenever $n > 1$ and $s > 0$. 
\end{proof} 

\begin{figure}[h]
\centering
\includegraphics[width=0.2\textwidth]{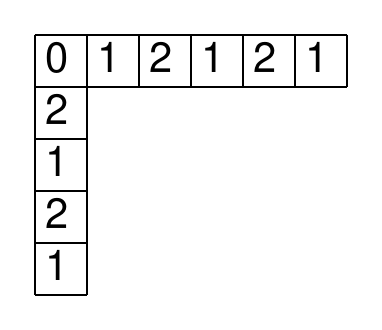}
\caption{The $\Gamma$-partition $\left( 6, 1^4 \right)$} \label{gammafig}
\end{figure}

We now prove a lemma that will be helpful in simplifying our discussion of Sprague-Grundy values of thick $\Gamma$-partitions. 


\begin{lemma} \label{diagzero}
Let $n, r, s$ be positive integers with $n > r$. Then $\left(n^r, r^s\right)$ is a $\mathcal P$-position, so its Sprague-Grundy value is zero. 
\end{lemma}
\begin{proof}
The result is true if $r = 1$ by~\cref{gammalem}, so suppose that $r> 1$. Whichever move the first player makes, the second player can respond with the opposite move, giving the position $\left((n-1)^{r-1}, (r-1)^{s-1}\right)$, which is also of the hypothesized form. By induction, this is a $\mathcal P$-position. Thus, the two intermediate positions are both $\mathcal N$-positions, so the original position is a $\mathcal P$-position, as desired. 
\end{proof}

In~\cref{tgstructfig} we know the entries in the right and down rectangles, thanks to~\cref{rowlem,2rowlem,rectlem}. By~\cref{diagzero} we know that the indicated diagonal cells are all filled with zeros. We need to determine the entries in the regions containing question marks. The entries in the region above the diagonal depend solely on the entries in the leftmost column of the right rectangle. Similarly, the entries below the diagonal depend only on the entries in the topmost column of the down rectangle. By~\cref{conjinvlem} we need only resolve one of these. We choose to focus on the entries above the diagonal. 

\begin{figure}[h] 
\centering
\includegraphics[width=.3\textwidth]{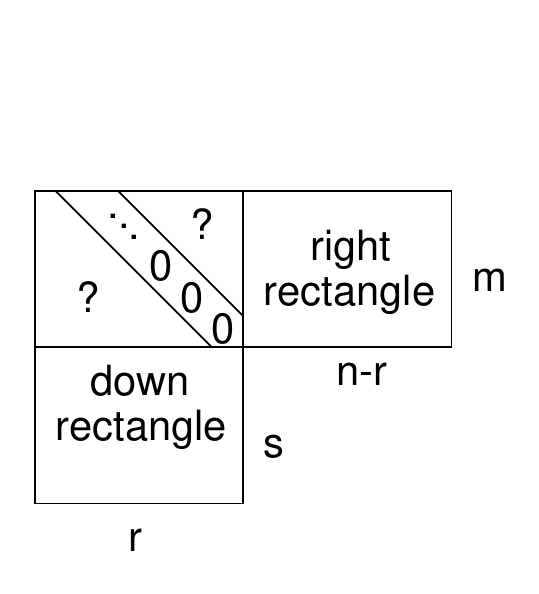}
\caption{The structure of $\left( n^m, r^s \right)$} \label{tgstructfig}
\end{figure}

First, suppose that $n = r+1$, as in the left image in \cref{tgstruct12fig}. The values in boldface are known to us; the following lemma tells us what the remaining values are. 

\begin{figure}[h]  
\centering
\includegraphics[width=.65\textwidth]{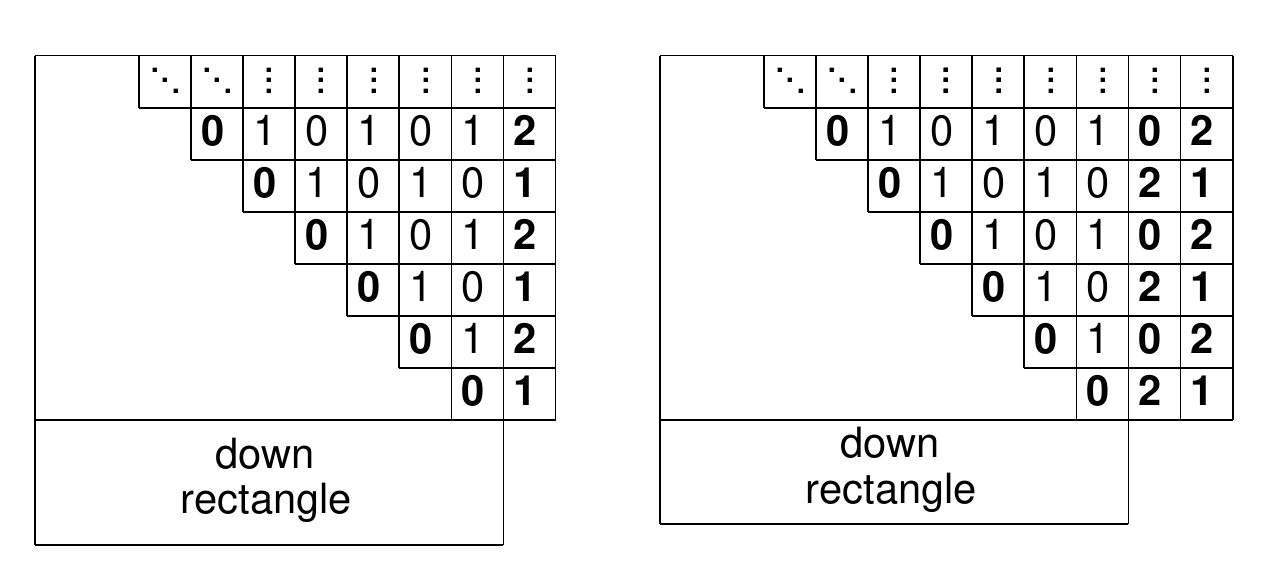}
\caption{Sprague-Grundy values for $\left( (r+1)^m, r^s \right)$ and $\left( (r+2)^m, r^s \right)$, where $m > r$} \label{tgstruct12fig}
\end{figure} 

\begin{lemma} \label{tgstruct1lem}
Let $r, \, m$, and $s$ be be positive integers with $m > r$. Then
\[ \sg\left((r+1)^m, r^s\right) = \left\{ \begin{array}{rl}
    0 & \mbox{if $r+m$ is even} \\
    1 & \mbox{if $r+m$ is odd}
\end{array} \right.. \]
\end{lemma}
\begin{proof}
We use induction on $m+r$. For the base step, suppose $m+r = 3$. Then $r = 1$ and $m = 2$, so 
\[ \sg \left( 2^2, 1^s \right) = \mex\left( \sg \left( 2, 1^s \right), \sg \left( 1^2 \right) \right) = \mex(0, 2) = 1.\] 
For the induction step, suppose that $m+r > 3$. If $m+r$ is even, then 
\begin{align*}
\sg\left((r+1)^m, r^s\right) 
= &\mex\left( \sg\left( (r+1)^{m-1}, r^s \right), \sg\left( r^m, (r-1)^s \right) \right) \\
= &\mex\left( 1, 1 \right) = 0. 
\end{align*}

If $m+r$ is odd, then 
\begin{align*}
\sg\left((r+1)^m, r^s\right) 
= &\mex\left( \sg\left( (r+1)^{m-1}, r^s \right), \sg\left( r^m, (r-1)^s \right) \right) \\
= &\left\{ \begin{array}{rl}
    \mex \left( 0, 0 \right) & \mbox{if $r > 1$}\\
    \mex \left( 0, 2 \right) & \mbox{if $r = 1$} 
\end{array} \right. \\
=& 1. 
\end{align*}
This establishes the result. 
\end{proof}

Next we consider the case when $n = r+2$. See the right image in \cref{tgstruct12fig}. 

\begin{lemma} \label{tgstruct2lem}
Let $r, m$ and $s$ be positive integers with $m > r$. Then 
\[ \sg\left((r+2)^m, r^s \right) = \left\{ \begin{array}{rl}
0 & \mbox{if $r+m$ is even} \\
1 & \mbox{if $r+m$ is odd}
\end{array} \right.. \]
\end{lemma} 
\begin{proof}
We use induction on $m+r$. For the base step, suppose $m+r = 3$. Then $r = 1$ and $m = 2$, so 
\[ \sg \left( 3^2, 1^s \right) = \mex\left( \sg \left( 3, 1^s \right), \sg \left( 2^2 \right) \right) = \mex(0, 0) = 1.\] 
For the induction step, suppose that $m+r > 3$. If $m+r$ is odd, then 
\begin{align*}
\sg\left((r+2)^m, r^s\right) 
=& \mex\left( \sg\left( (r+2)^{m-1}, r^s \right), \sg\left( (r+1)^m, (r-1)^s \right) \right) \\
=& \mex\left( 0, 0 \right) = 1. 
\end{align*}

If $m+r$ is even, then 
\begin{align*}
\sg\left((r+2)^m, r^s\right) 
&= \mex\left( \sg\left( (r+2)^{m-1}, r^s \right), \sg\left( (r+1)^m, (r-1)^s \right) \right) \\
&= \left\{ \begin{array}{rl}
    \mex \left( 1, 1 \right) & \mbox{if $r > 1$}\\
    \mex \left( 1, 2 \right) & \mbox{if $r = 1$} 
\end{array} \right. \\
&=  0.
\end{align*}
This establishes the result. 
\end{proof}

Finally we consider the case when $n > r+2$. We treat the even and odd cases for $n-r$ separately. 

\begin{figure}[h] 
\centering
\includegraphics[width=.9\textwidth]{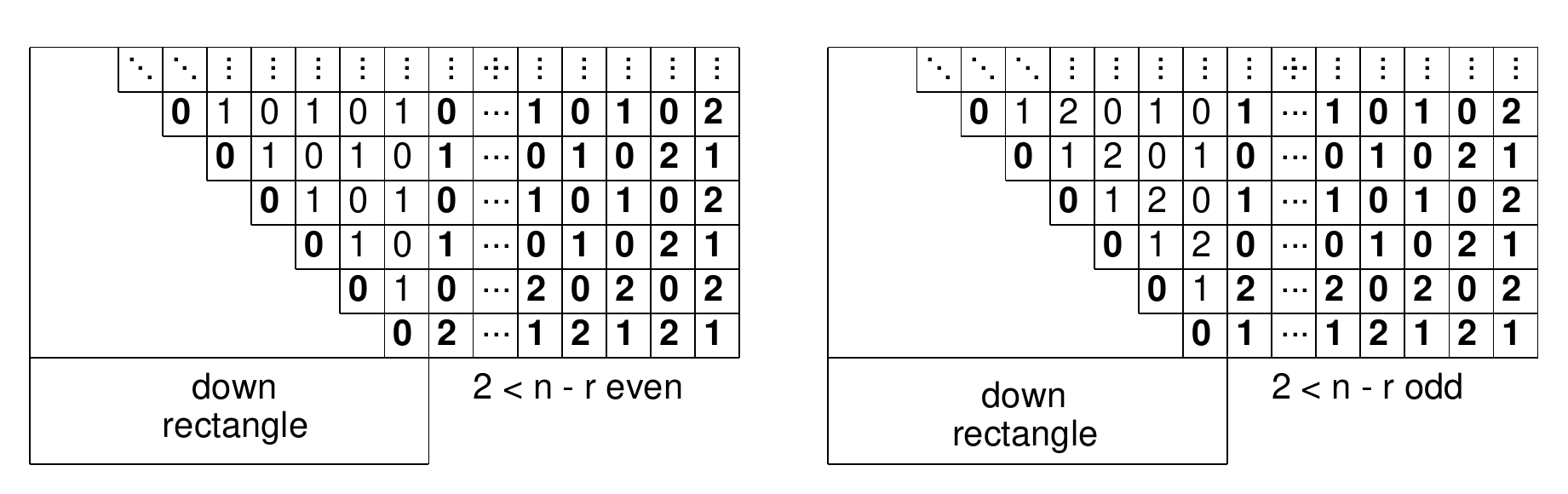}
\caption{Sprague-Grundy values of $\left( (n)^m, r^s \right)$ where $m > r$} \label{tgstruct3fig}
\end{figure}

\begin{lemma} \label{tgstruct3lem}
Let $m, n, r$, and $s$ be positive integers with $n > r + 2$ and $m > r$. If $n-r$ is even, then 
\[ \sg\left( n^m, r^s \right) 
= \left\{ \begin{array}{rl} 
    0 & \mbox{if $r+m$ is even} \\
    1 & \mbox{if $r+m$ is odd}
\end{array} \right..\]
If $n-r$ is odd, then 
\[ \sg\left( n^m, r^s \right) 
= \left\{ \begin{array}{rl} 
    0 & \mbox{if $m=r$, or $m-r$ is odd and $m > 3$} \\
    1 & \mbox{if $m=r+1$, or $m-r$ is even and $m > 4$} \\
    2 & \mbox{if $m = r+2$}
\end{array} \right..\]
\end{lemma}
\begin{proof}
First suppose that $n-r$ is even. We proceed by induction on $m+r$. If $m+r = 3$, then $m = 2$ and $r = 1$ so
\[ \sg\left( n^2, 1^s \right) 
= \mex \left( \sg \left( n, 1^s \right), \sg \left( (n-1)^2 \right) \right) \\
= \mex(0,0) 
= 1.\]
For the induction step, suppose that $m+r > 1$ and $m+r$ is even. Then
\[ \sg\left( n^m, r^s \right) \\
= \mex \left( \sg \left( n^{m-1}, r^s \right), \sg \left( (n-1)^m, (r-1)^s \right) \right) \\
= \mex \left( 1, 1 \right) 
= 0.
\]
Now suppose that $m+r$ is odd. Then 
\[ \sg\left( n^m, r^s \right) \\
= \mex \left( \sg \left( n^{m-1}, r^s \right), \sg \left( (n-1)^m, (r-1)^s \right) \right) \\
= \mex \left( 0, 0 \right) 
= 1. \]
This concludes the case when $n-r$ is even. 

Suppose $n-r$ is odd. We again use induction on $m+r$. If $m+r = 3$, then $m = 2$ and $r = 1$ so 
\[ \sg\left( n^2, 1^s \right) 
= \mex\left( \sg\left( n, 1^s \right), \sg\left( (n-1)^2 \right) \right) 
= \mex(0, 2) 
= 1. \]
For the induction step, suppose that $m+r > 3$. Suppose $m+r$ is even. Then \begin{align*}
    \sg\left( n^m, r^s \right) 
&= \mex \left( \sg \left( n^{m-1}, r^s \right), \sg \left( (n-1)^m, (r-1)^s \right) \right) \\
&= \left\{ \begin{array}{rl} 
    \mex(1, 0) & \mbox{if $m-r = 2$} \\
    \mex(0, 0) & \mbox{if $m-r > 2$} 
    \end{array} \right. \\
&= \left\{ \begin{array}{rl} 
    2 & \mbox{if $m-r=2$} \\
    1 & \mbox{if $m-r>2$} 
\end{array} \right..
\end{align*} 
Now suppose $m+r$ is odd. Then \begin{align*}
    \sg\left( n^m, r^s \right)
&= \mex \left( \sg \left( n^{m-1}, r^s \right), \sg \left( (n-1)^m, (r-1)^s \right) \right) \\
&= \left\{ \begin{array}{rl} 
    \mex(0, 2) & \mbox{if $m-r = 1$} \\
    \mex(2, 1) & \mbox{if $m-r = 3$} \\
    \mex(1, 1) & \mbox{if $m-r > 3$} 
    \end{array} \right. \\
&= \left\{ \begin{array}{rl} 
    0 & \mbox{if $m-r=3$} \\
    1 & \mbox{if $m-r = 1$ or $m-r>3$} 
\end{array} \right.,
\end{align*} 
as desired.
\end{proof}

We summarize \cref{tgstruct1lem,tgstruct2lem,tgstruct3lem} in the following theorem. 

\begin{theorem}
Let $n, m, r$, and $s$ be be positive integers with $m > r$ and $n > r$. If $n = r+1$, if $n=r+2$, or if $n > r+2$ and $n-r$ is even, then
\[ \sg\left(n^m, r^s\right) = \left\{ \begin{array}{rl}
    0 & \mbox{if $r+m$ is even} \\
    1 & \mbox{if $r+m$ is odd}
\end{array} \right.. \]
If $n > r+2$ and $n-r$ is odd, then 
\[ \sg\left( n^m, r^s \right) 
= \left\{ \begin{array}{rl} 
    0 & \mbox{if $m=r$ or if $m-r$ is odd and $m > 3$} \\
    1 & \mbox{if $m=r+1$ or if $m-r$ is even and $m > 4$} \\
    2 & \mbox{if $m = r+2$}
\end{array} \right..\]
\end{theorem}

\subsection{Computing Sprague-Grundy values of arbitrary partitions}

For impartial games in general, even just classifying a position to be winning or losing one can be notoriously difficult task, as recursive computation of (\ref{eq:SG}) can quicly lead to a combinatorial explosion.
This is due to the fact that such games (including LCTR) usually admit exponentially many different plays from a given starting position.

This section concerns the computational difficulty of computing the Sprague-Grundy values of partitions not necessarily included in one of the families analyzed in \cref{sec:LCTR}. 
The core observation towards fast  computation of those values is bounding the number of possible positions reachable from of the given initial position.

Due to the notation developed in \cref{sec:LCTR} (also see Figures~\ref{rowfig}-\ref{tgstruct3fig}), it is clear that every box of the input Young diagram gives rise to another  Young diagram, reachable from the initial one. 

\begin{corollary}\label{cor:subpositions}
For any partition of $n$ there exist at most $n$ distinct partitions reachable from it. 
\end{corollary}

We remark that, for a given partition $\lambda$ of $n$, the number of  distinct partitions reachable from it may be much smaller. 
For instance, the partition of $15$ from \cref{fig:fer} can be played in $29$ different ways, however it only contains $11$ distinct positions reachable from it, namely
\[\{
(),
(1),
(1^2),
(2),
(2,1), 
(2, 1^2), 
(2^2, 1), 
(3,1^2),
(3,2,1^2),
(3^2,2,1^2),
(4, 2^2, 1)
\}.\]
An even more extreme example is a partition of $21$ from \cref{stairfig}. Indeed, the staircase $s_6$ can be played in $64$ different ways, but admits only $6$ distinct positions reachable from it (namely $\{(),(1),s_2,s_3,s_4,s_5\}$).
The next lemma shows that the bound from \cref{cor:subpositions} is tight, where the tightness is attained whenever the input partition is a rectangle. 

\begin{lemma}\label{lem:max-subpositions}
For a rectangle $\lambda= (n^m)$ there exist exactly $mn$ distinct partitions reachable from it. 
Any partition of $N$ which is not a rectangle admits less then $N$ distinct partitions reachable from it.
\end{lemma}

\begin{proof}
When considering the number of distinct partitions reachable from the rectangle $\lambda$, one should disregard $\lambda$ itself, 
while on the other hand, one should not forget about the empty partition $()$.
It is hence enough to  show that every distinct square from the Young diagram of $\lambda$ gives rise to a unique Young diagram. 
In fact, due to the structure of $\lambda= (n^m)$ we can explicitly express the diagram corresponding to any square on coordinates $(i,j)$, simply as a Young diagram $((n-i)^{m-j})$.
Since $n$ and $m$ are fixed, any distinct coordinate clearly give rise to a distinct Young diagram, as desired. 

Consider now the second part of the statement and let $\lambda$ be an arbitrary partition of $N$ which admits $N$ subpositions reachable from it. 
By \cref{cor:subpositions} (and treating the terminal position $()$ separately) we infer that the $n-1$ squares of Young diagram $\lambda$ give rise to distinct subpositions of $\lambda$. 
In particular, there can only be one ``corner square", that is, a subposition which gives rise to the partition $(1)$.
This in turn implies that there exist positive integers $a$ and $b$ such that $\lambda=(a^b)$,  as desired.
\end{proof}

\subsubsection{Description of \cref{alg:pseudo}}
The above implies that, in order to obtain the Sprague-Grundy value of a given starting Young diagram $\lambda$ on $N$ boxes, it suffices to compute the Sprague-Grundy value of as little as $N-1$ smaller Young diagrams\footnote{Here we assume that the terminal position $()$ is precomputed.}, allowing for a computation time which is linear in $N$.
 
\begin{algorithm}[h]
\caption{An outline of the code from \cite{jelena_code} which computes $SG(\lambda)$, for any input partition $\lambda$. 
\label{alg:pseudo}}
\begin{algorithmic}
\State $\mathcal{D} \gets \{():0\}$
\Procedure{ComputeSG}{$\lambda$}
    \If{$\lambda \notin \mathcal D$}
        \State $\mathcal D[\lambda]\gets \mex(\textsc{ComputeSG}(T(\lambda)),\textsc{ComputeSG}(L(\lambda)))$
    \EndIf
    \State \textbf{return} $\mathcal D [\lambda]$ 
\EndProcedure
\end{algorithmic}
\end{algorithm}

\begin{observation}\label{obs:fast-computingSG}
Let $\lambda$ be an initial partition on $n$ squares. 
Then one can compute the Sprague-Grundy value (with respect of the LCTR game) in $O(n)$. 
\end{observation} 
The above observation naturally leads to \cref{alg:pseudo}, where the data structure $\mathcal D$ should be interpreted as a dictionary.
Keys of $\mathcal D$ consist of the possible partitions reachable from $\lambda$; those are stored together with the corresponding Sprague-Grundy values.
The algorithm recursively computes the Sprague-Grundy value by using \cref{eq:SG} (also see \cite{bouton1901nim} for historic reference).
We emphasise that the implementation in \cref{alg:pseudo}, and correspondingly in \cite{jelena_code}, dynamically stores the computed Sprague-Grundy values of distinct positions reachable from $\lambda$, so that each of them is computed exactly once.
For the precise implementation of the computation summarized in  \cref{alg:pseudo}, written in Python, we refer the reader to \cite{jelena_code}.

 \section{Conclusion} \label{conc}
This paper introduces LCTR, a new impartial game on partitions, in which players take turns removing either the Left Column or the Top Row of the corresponding Young diagram.
We establish that the Sprague-Grundy value of any partition is at most $2$, and determine Sprague-Grundy values for several infinite families of partitions.
By analyzing the sparseness of the positions reachable from a given partition $\lambda$ of $n$, we use a dynamic programming approach which determines the corresponding Sprague-Grundy value in $O(n)$ time.
It would be interesting to study further, both (i) the sparseness of the positions reachable from a given partition, as well as (ii) the computational difficulty of LCTR.
\begin{enumerate}[(i)]
    \item \cref{lem:max-subpositions} describes the tight upper-bound on the number of the positions reachable from a given partition $\lambda$ of $n$. 
    The statement is supported by an infinite family attaining the bound $n+1$.
    As some partitions clearly do not match this bound (even in an asymptotic way, see \cref{fig:fer,stairfig}), it would be interesting to identify the \emph{lower-bound} for the number of distinct positions reachable from a given partition $\lambda$ of $n$. 
    \item Although the linear computational complexity achieved by \cref{alg:pseudo} seems to be very efficient, it still analyzes all reachable positions, which might not be necessary. Indeed, several impartial games admit an elegant winning strategy despite allowing for huge number of reachable positions. 
    It would be interesting if the computational complexity of LCTR could be improved even further. 
\end{enumerate}





\printbibliography

\end{document}